\documentclass[11pt]{article}
\usepackage[margin=1in]{geometry} 
\usepackage{amssymb,amsfonts,amsmath,bbm,mathrsfs,stmaryrd,mathtools}
\usepackage{xcolor}
\usepackage{url}

\usepackage{enumerate}

\usepackage{hyperref}
\hypersetup{colorlinks,
             linkcolor=black!75!red,
             citecolor=blue,
             pdftitle={},
             pdfproducer={pdfLaTeX},
             pdfpagemode=None,
             bookmarksopen=true
             bookmarksnumbered=true}

\usepackage{tikz}
\usetikzlibrary{arrows,calc,decorations.pathreplacing,decorations.markings,intersections,shapes.geometric,through,fit,shapes.symbols,positioning,decorations.pathmorphing}

\usepackage{braket}

\usepackage[amsmath,thmmarks,hyperref]{ntheorem}
\usepackage{cleveref}

\creflabelformat{enumi}{#2(#1)#3}

\crefname{section}{Section}{Sections}
\crefformat{section}{#2Section~#1#3} 
\Crefformat{section}{#2Section~#1#3} 

\crefname{subsection}{\S}{\S\S}
\AtBeginDocument{%
  \crefformat{subsection}{#2\S#1#3}%
  \Crefformat{subsection}{#2\S#1#3}% 
}

\crefname{subsubsection}{\S}{\S\S}
\AtBeginDocument{%\Crefformat{notation}{#2Notation~#1#3} 

\crefname{table}{table}{tables}
\crefformat{table}{#2table~#1#3} 
\Crefformat{table}{#2Table~#1#3}

\crefname{lemma}{lemma}{lemmas}
  \crefformat{subsubsection}{#2\S#1#3}%
  \Crefformat{subsubsection}{#2\S#1#3}% 
}

\theoremstyle{plain}

\newtheorem{lemma}{Lemma}[section]
\newtheorem{proposition}[lemma]{Proposition}
\newtheorem{corollary}[lemma]{Corollary}
\newtheorem{theorem}[lemma]{Theorem}
\newtheorem{conjecture}[lemma]{Conjecture}

\theoremstyle{nonumberplain}

\theoremstyle{plain}
\theorembodyfont{\upshape}
\theoremsymbol{\ensuremath{\blacklozenge}}

\newtheorem{definition}[lemma]{Definition}
\newtheorem{example}[lemma]{Example}
\newtheorem{remark}[lemma]{Remark}

\crefname{definition}{definition}{definitions}
\crefformat{definition}{#2definition~#1#3} 
\Crefformat{definition}{#2Definition~#1#3} 

\crefname{ex}{example}{examples}
\crefformat{example}{#2example~#1#3} 
\Crefformat{example}{#2Example~#1#3} 

\crefname{remark}{remark}{remarks}
\crefformat{remark}{#2remark~#1#3} 
\Crefformat{remark}{#2Remark~#1#3} 

\crefname{convention}{convention}{conventions}
\crefformat{convention}{#2convention~#1#3} 
\Crefformat{convention}{#2Convention~#1#3} 

\crefname{notation}{notation}{notations}
\crefformat{notation}{#2notation~#1#3} 

\crefformat{lemma}{#2lemma~#1#3} 
\Crefformat{lemma}{#2Lemma~#1#3} 

\crefname{proposition}{proposition}{propositions}
\crefformat{proposition}{#2proposition~#1#3} 
\Crefformat{proposition}{#2Proposition~#1#3} 

\crefname{corollary}{corollary}{corollaries}
\crefformat{corollary}{#2corollary~#1#3} 
\Crefformat{corollary}{#2Corollary~#1#3} 

\crefname{theorem}{theorem}{theorems}
\crefformat{theorem}{#2theorem~#1#3} 
\Crefformat{theorem}{#2Theorem~#1#3} 

\crefname{enumi}{}{}
\crefformat{enumi}{(#2#1#3)}
\Crefformat{enumi}{(#2#1#3)}

\crefname{assumption}{assumption}{Assumptions}
\crefformat{assumption}{#2assumption~#1#3} 
\Crefformat{assumption}{#2Assumption~#1#3} 

\crefname{equation}{}{}
\crefformat{equation}{(#2#1#3)} 
\Crefformat{equation}{(#2#1#3)}

\numberwithin{equation}{section}
\renewcommand{\theequation}{\thesection-\arabic{equation}}

\theoremstyle{nonumberplain}
\theoremsymbol{\ensuremath{\blacksquare}}

\newtheorem{proof}{Proof}
\newcommand\pf[1]{\newtheorem{#1}{Proof of \Cref{#1}}}

\newcommand\bF{{\mathbb F}}

\newcommand\bQ{{\mathbb Q}}

\newcommand\bZ{{\mathbb Z}}

\DeclareMathOperator{\id}{id}

\DeclareMathOperator{\ord}{\mathrm{ord}}

\newcommand\numberthis{\addtocounter{equation}{1}\tag{\theequation}}

\newcommand{\qedhere}{\mbox{}\hfill\ensuremath{\blacksquare}}

\title{Quadratic rotation symmetric Boolean functions}
\author{Alexandru Chirvasitu and Thomas W. Cusick}

\begin{document}

\date{}

\newcommand{\Addresses}{{% additional braces for segregating \footnotesize
  \bigskip
  \footnotesize

  \textsc{Department of Mathematics, University at Buffalo}
  \par\nopagebreak
  \textsc{Buffalo, NY 14260-2900, USA}  
  \par\nopagebreak
  \textit{E-mail address}: \texttt{achirvas@buffalo.edu}

  \medskip

  \textsc{Department of Mathematics, University at Buffalo}
  \par\nopagebreak
  \textsc{Buffalo, NY 14260-2900, USA}  
  \par\nopagebreak
  \textit{E-mail address}: \texttt{cusick@buffalo.edu}
  
  % % \medskip
  % % 
  % % \textsc{Department of Mathematics, INSTITUTION}
  % % \par\nopagebreak
  % % \textsc{ADDRESS}
  % % \par\nopagebreak
  % % \textit{E-mail address}: \texttt{??}
  % % 

}}

\maketitle

\begin{abstract}
  Let $(0, a_1, \ldots, a_{d-1})_n$ denote the function $f_n(x_0, x_1, \ldots, x_{n-1})$ of degree $d$ in $n$ variables generated by the monomial $x_0x_{a_1} \cdots x_{a_{d-1}}$ and having the property that $f_n$ is invariant under cyclic permutations of the variables. Such a function $f_n$ is called monomial rotation symmetric (MRS). Much of this paper extends the work on quadratic MRS functions in a $2020$ paper of the authors to the case of binomial RS functions, that is sums of two quadratic MRS functions.  There are also some results for the sum of any number of quadratic MRS functions.
\end{abstract}


\vspace{.5cm}


%\tableofcontents

%%%%%%%%%%%%%%%%%%%%%%%%%%%%%%%%%%%%%%%%%%%%%%%%%%%%%%%%%%%%%%%%%%%%%%%%%%%%%
%%%%%%%%%%%%%%%%%%%%%%%%%%%%%%%%%%%%%%%%%%%%%%%%%%%%%%%%%%%%%%%%%%%%%%%%%%%%%

% % %%%%%%%%%%%%%%%%%%%%%%%%%%%%%%%%%%%%%%%%%%%%%%%%%%%%%%%%%%%%%%%%%%%%%%%%%%%%%
% % \subsection*{Declarations of interest: none}
% %

%%%%%%%%%%%%%%%%%%%%%%%%%%%%%%%%%%%%%%%%%%%%%%%%%%%%%%%%%%%%%%%%%%%%%%%%%%%%%
%%%%%%%%%%%%%%%%%%%%%%%%%%%%%%%%%%%%%%%%%%%%%%%%%%%%%%%%%%%%%%%%%%%%%%%%%%%%%
\section{Introduction}\label{se:int}

Applications of Boolean functions to cryptography and coding theory have been studied for decades (see the books \cite{carl, cs-bk}).  A Boolean function in $n$ variables is rotation symmetric (RS) if it is invariant under powers of the cyclic shift $\rho(x_1, \ldots..., x_n) = (x_2, \ldots, x_n, x_1)$. An RS function is said to be monomial rotation symmetric (MRS) if it is generated by applying powers of $\rho$ to a single monomial.  Since the definition of RS functions was introduced in $1999$ \cite{hash}, the set of RS functions has been shown to be an extremely fertile source of functions with useful cryptographic properties (\cite[Section 6.2]{cs-bk} gives information about this).  This paper continues the detailed investigation of the quadratic RS functions which was initiated in \cite{cc-bal}.

We will need some basic definitions for a Boolean function $b(x_1, \ldots, x_n)=b(\vec{x})=b_n$ in $n$ variables. We define $F_n$ to be the vector space of dimension $n$ over the finite field $\bF_2$ with $2$ elements. When addition mod $2$ is clear from the context we use $+,$ but if we wish to emphasize the fact that addition is being done mod $2$ we will use $\oplus.$ We also use $\oplus$ for the mod $2$ addition of two binary $m$-tuples.  We use $deg~ b$ for the degree of $b$ so $deg ~b_n = n.$ The (Hamming) weight $wt(b_n)$ is the number of $1$'s in the value set of $b_n.$ Function $b_n$ is said to be {\it balanced} if $wt(b_n) =2^{n-1}.$

We shall need the notions of the {\it Walsh spectrum} and {\it nonlinearity} of a Boolean function $b_n,$ both of which are extremely important in cryptography.  The Walsh spectrum is the set of the distinct (integer) values of the {\it Walsh transform} (sometimes called the Hadamard transform in the literature) $W(b_n)(\vec{y}),$ which is defined by

\begin{equation}
\label{Wdef}
W(b_n)(\vec{y}) =  \sum_{\vec{x} \in F_n}   (-1)^{b(\vec{x})+ \vec{x} \cdot \vec{y}}.
\end{equation}
Here $\vec{x} \cdot \vec{y}$ is the usual dot product of two vectors, computed in 
$\bF_2.$
The integers $W(b_n)(\vec{0})$ are related to $wt(b_n)$ because of the elementary identity
\begin{equation}
\label{wtWlsh}
W(b_n)(\vec{0}) = 2^n - 2wt(b_n).
\end{equation}
The nonlinearity of $b_n$ is defined in terms of the {\it distance} $d(f,g)$ between two Boolean functions $f$ and $g,$ which is given by
\begin{equation*}
d(f, g) = wt(f \oplus g).
\end{equation*}
The nonlinearity $N(f)$ of $b$ is
\begin{equation*}
N(b) =  \min_{\text{deg~g} \leq 1}  d(b, g).
\end{equation*}
As usual, we say that a function $g$ with $deg ~g \leq 1$ is an {\em affine} function.

It is important that Boolean functions used in cryptography have large nonlinearity. 
It is an old result  \cite[Theorem 2.31, p. 22]{cs-bk} that any $b_n$ satisfies the inequality
\begin{equation}
\label{nlbd}
N(b_n) \leq 2^{n-1} - 2^{(n-2)/2}.
\end{equation}
 If $n$ is even, then there exist functions $f$ for which $N(f)$ attains the maximum
 value in \Cref{nlbd}, and these functions are called {\it bent}. 
 
Equality in \Cref{nlbd} is obviously impossible if $n$ is odd, but a sharp upper bound for $N(b_n)$ in this case is not known
for general odd $n.$   The nonlinearity  $2^{n-1} - 2^{(n-1)/2},$ and nothing larger, can be attained for any odd $n$ with a suitably chosen quadratic function in $n$ variables (see \cite[Lemma 5, p. 429]{kph} for a more general result), and so this value is often referred to as the {\it quadratic bound} in the literature.

The Walsh spectrum of $b_n$  determines the nonlinearity of $b_n$ because of the following lemma.
\begin{lemma}
\label{nlf}
The Walsh spectrum of a Boolean function $b_n$ determines the nonlinearity by

$$N(b_n) = 2^{n-1} - \frac{1}{2} \max_{\vec{y} \in F_n} |W(b_n)(\vec{y})|.$$
\end{lemma}
\begin{proof}
This is well known, see \cite[Theorem 2.21, p. 17]{cs-bk}.
\end{proof}
It follows from \Cref{nlf} and \Cref{nlbd} that 
$\max_{\vec{y} \in F_n} |W(b_n)(\vec{y})| \geq 2^{n/2}$ for all $\vec{y}$ and in fact equality holds for all $\vec{y}$ if and only if $b_n$ is bent, by the well known Parseval equation (see \cite[Corollary 2.23, p. 17]{cs-bk})
\begin{equation} \label{pars}
 \sum_{\vec{y} \in F_2} W(b_n)(\vec{y})^2 = 2^{2n}.
\end{equation}

A Boolean function $b_n$ is said to be {\it plateaued} if every integer in the Walsh spectrum is either $0$ or $\pm M,$ where $M$ is some positive integer. It is well known that every quadratic function is plateaued (a proof is given in \cite[p. 260]{carl}). We follow the terminology of \cite[p. 1308]{cc-bal} and say that a quadratic function $b_n$  is {\it v-plateaued} if every Walsh value is either $0$ or $\pm 2^{(n+v)/2}.$ The same terminology is used in \cite[p. 166]{amt}, but with $s$ instead of $v.$ We shall call $v$ the {\it plateau parameter} and of course for any specific Boolean quadratic function $b_n$ the plateau parameter $v=v(n)$ will depend on $n.$

Two Boolean functions $f(\vec{x})$ and $g(\vec{x})$ in $n$ variables are said to be {\it affine equivalent} if there exists an invertible matrix $A$ with entries in 
$\bF_2$ and ${\bf b} \in F_n$ such that $ f((\vec{x}))=g(A(\vec{x})\oplus {\bf b}).$
In general, determining whether two Boolean functions are affine equivalent is difficult, even in the simplest cases.  However, there is a simple test for the quadratic functions in \cite[Lemma 1.1]{cc-bal}, which we repeat below.
\begin{lemma}
\label{quadeq}
Two quadratic Boolean functions $f$ and $g$ in $n$ variables are affine equivalent if and only if $wt(f) = wt(g)$ and $N(f) = N(g).$
\end{lemma}
\begin{proof}
This result has been well known for a long time, but there does not seem to be a proof in the literature before the one given in
 \cite[Lemma 2.3, p. 5068]{cus_aff-cub}.
\end{proof}

It is worth noticing that \Cref{wtWlsh}, \Cref{quadeq} and the fact that all quadratic functions are $v-$plateaued with $0 \leq v \leq n-2$ ($v=0$ occurs if and only if the function is bent) imply the following lemma.

\begin{lemma} \label{quadwt}
All possible weights for a quadratic RS function in $n$ variables are $2^{n-1}$
(if the function is balanced) and $2^{n-1} \pm 2^j,~(n/2)-1 \leq j \leq n-2.$
\end{lemma}

Some results for quadratic RS functions $b_n$ in this paper apply only if $n$ is large enough so that bent functions do not occur (see the discussion of this point in \cite[p. 1308]{cc-bal}); we will not mention this issue again in stating the results below. Generally assuming $n \geq 2t+1$ (which eliminates the bent function $(0,t)_{2t}$) is what is needed if the function being considered contains the monomial $(0,t)_n.$

It is clear from \Cref{quadeq} and \Cref{quadwt} that the affine equivalence class of any quadratic RS function $b_n$ is determined if we know whether or not $b_n$ is balanced, we know the parameter $v(n)$ and we know the $\pm$ signs in \Cref{quadwt}.  If $b_n$ is balanced, then we only need 
$v(n)$ to determine $N(b_n)$ by \Cref{nlf}; if $b_n$ is not balanced, then we need $v(n)$ to determine both $wt(b_n)$ and $N(b_n).$

A complete solution to the problem of finding the equivalence class to which any MRS function belongs was given in \cite[Section 4]{cc-bal} and a formula for the number of affine equivalence classes for the quadratic MRS functions $(0,t)_n$ was given in \cite[Theorem 4.3]{cc-bal}. In particular,  \cite[Lemma 4.2]{cc-bal} gives a precise description of the values of $n \geq 2t+1$ for which $(0,t)_n$ is balanced. Also the period length $2t$ for the list of plateau parameters $\{v(n): n \geq 2t+1\}$ and details about the structure (including the largest two values) of that period were given in \cite[Lemma 4.1]{cc-bal}. Further, 
\cite[Theorem 5.2]{cc-bal} (which we repeat as \Cref{th:prd} below) proves that the list of $v(n)$ values is periodic for any quadratic RS function $b_n,$ and gives the largest
 $v(n).$  We use notation $V(b)$ for the period length. 

It is important to recall that the sequence of weights for {\it any} Boolean RS function satisfies a linear recurrence with integer coefficients.  A proof of this is in \cite{cus_rec}, where a {\it rules matrix} is defined with the property that the recursion polynomial associated with the linear recurrence for a given RS function can be chosen as either the minimal polynomial or the characteristic polynomial of the rules matrix.  A detailed description of an algorithm for computing these rules matrices for any RS function is in \cite{cus_rec-xv}, along with a Mathematica program which computes the rules matrix.  In both of these papers, specializing to the case of quadratic functions dramatically simplifies the algorithm and its proof.  A full description of the rules matrices for MRS functions is in \cite[Theorem 3.1]{cc-bal} and a similar description for the more complicated case of binomial RS quadratic functions is in \Cref{se:binrules}.

In \Cref{se:binom} of the present paper, we extend the results in  
\cite[Example 5.1]{cc-bal} on binomial quadratic RS functions $(0,j)_n + (0,i)_n, ~j<i,$ in two ways. \Cref{th:whenbinombal} is a precise description of the values of $n$ for which the binomial functions are balanced and \Cref{th:gij} gives the simple formula 
$2\cdot \mathrm{lcm}(i+j,i-j)$ for the period length of the list $v(n).$
\Cref{se:binom} also gives various results on when sums of $3$ or more quadratic RS functions are balanced. \Cref{se:grt2} explains an algorithm for computing the period of the list of $v(n)$ values for the sum of any number of quadratic RS functions.  \Cref{se:scaled1} generalizes  
\cite[Proposition 2.8]{cc-rec}, used in the study of the rules matrices for MRS functions, with the expectation of applications to the study of rules matrices for binomial RS functions, which are described in detail in \Cref{se:binrules}.

%%%%%%%%%%%%%%%%%%%%%%%%%%%%%%%%%%%%%%%%%%%%%%%%%%%%%%%%%%%%%%%%%%%%%%%%%%%%%
\subsection*{Acknowledgements}

A.C. is partially supported by NSF grant DMS-2001128. 

%%%%%%%%%%%%%%%%%%%%%%%%%%%%%%%%%%%%%%%%%%%%%%%%%%%%%%%%%%%%%%%%%%%%%%%%%%%%%
%%%%%%%%%%%%%%%%%%%%%%%%%%%%%%%%%%%%%%%%%%%%%%%%%%%%%%%%%%%%%%%%%%%%%%%%%%%%%
\section{Preliminaries}\label{se.prel}

We shall use the formula for $v(n)$ in the lemma below. We use $\gcd_2$ to indicate a gcd taken mod $2.$
\begin{lemma}
 \label{vvals1}
 Given the quadratic RS Boolean function
 \begin{equation} \label{Qa}
   Q=Q(a_1, a_2, \ldots, a_{[(n-1)/2]})=\sum_{i=1}^{[(n-1)/2]} a_i (0,i)_n
 \end{equation}
 with each $a_i$ in $\{0,1\}$ define
 \begin{equation}
 \label{fncAn}
A_n(x) = \sum_{i=1}^{[(n-1)/2]} a_i (x^{i} +x^{n-i}).
 \end{equation}
 Then the $v$-values for $Q$ are given by 
 \begin{equation}
 \label{vng}
 v(n) = \mathrm{deg} ~\mathrm{gcd_2}(x^n - 1, A_n(x)),
 \end{equation}
 where the greatest common divisor is taken mod $2.$
 \end{lemma}
\begin{proof}
A proof is given in \cite[p. 269]{amt}, where $s$ is used instead of our  notation $v(n).$ The trace form of our functions \Cref{Qa} is used in \cite{amt} but by
\cite[Remark 1, p. 268]{amt} this does not matter. There are also four references, dating back to $2006,$ to this well-known lemma given in \cite[p. 269]{amt}.
\end{proof}

For a prime $p$, the {\it $p$-adic valuation} $\nu_p(n)$ of an integer $n$ is the largest $\nu \in \bZ_{\ge 0}$ such that $p^{\nu}$ divides $n$. The term applies more generally to primes in unique factorization domains. For a field $\bF$, for instance, and an irreducible polynomial $p(x)\in \bF[x]$, we can speak of the $p$-adic valuation $\nu_p(-)$:
\begin{equation*}
  \text{for }q\in \bF[x],\ \nu_p(q) = \max\{\nu\in \bZ_{\ge 0}\ |\ p^{\nu}\text{ divides }q\}.
\end{equation*}

We will frequently work with rings $\bF[x^{\pm 1}]$ of {\it Laurent} polynomials instead. In particular, if we write function $Q$ in \Cref{fncAn} in the form
\begin{equation} \label{AnJ}
  A_n(x)=\sum_{i=1}^{J:=J(Q)} a_i(x^{i}+x^{n-i}),
\end{equation}
where $J$ is as small as possible,
then we can define the corresponding Laurent polynomial by
\begin{equation}\label{fncA}
  A(x)=\sum_{i=1}^{J:=J(Q)} a_i(x^{i}+x^{-i}).
\end{equation}
Now we can state the next theorem (note we assume $n \geq 2J+1$ in order to avoid the short function $(0, J)_{2J}$).
\begin{theorem} \label{th:prd}
 Given any $Q=Q(a_1, a_2, \ldots, a_{[(n-1)/2]}),$ the period (of length $V(Q)$) for the list of $v$-values beginning with $n=2J(Q)+1$ has a unique largest integer 
 $2J= 2 \max_{a_i \neq 0} i=v(V(Q)).$ The entries in the period are symmetric around this largest integer in the sense that  $v(V(Q)i+r) = v(V(Q)i-r)$ for each 
 $r = 1, 2, ...$ for which both sides of the equation are defined and for each $i$ satisfying $1 \leq j \leq V(Q).$ This gives 
 \begin{equation} \label{vnA}
  v(n) = \mathrm{deg} ~\mathrm{gcd_2}(x^n - 1, A(x)),
 \end{equation}
\end{theorem}
\begin{proof}
This is \cite[Theorem 5.2, p. 1012]{cc-bal} restated in our notation. We can replace \Cref{vng} with \Cref{vnA} because in every splitting field of $x^n-1$ over $\bF_2$ $x^n$ is identically $1,$ so we can remove the $n$ from $A_n(x).$
\end{proof}

%%%%%%%%%%%%%%%%%%%%%%%%%%%%%%%%%%%%%%%%%%%%%%%%%%%%%%%%%%%%%%%%%%%%%%%%%%%%%
%%%%%%%%%%%%%%%%%%%%%%%%%%%%%%%%%%%%%%%%%%%%%%%%%%%%%%%%%%%%%%%%%%%%%%%%%%%%%
\section{Binomial rotation symmetric functions}\label{se:binom}

%%%%%%%%%%%%%%%%%%%%%%%%%%%%%%%%%%%%%%%%%%%%%%%%%%%%%%%%%%%%%%%%%%%%%%%%%%%%%
\subsection{Periods}\label{subse:per}

In light of \cite[Theorem 5.2]{cc-bal} (and its proof), we can revisit the discussion of the periods $V(g),$ where $g = g_i = (0,1)_n + (0, i)_n$ in \cite[Example 5.1]{cc-bal}. We begin with a more general result.

\begin{theorem}\label{th:gij}
  Let $1\le j<i$ be two positive integers. The list of $v$-values for the sum $(0,j)_n + (0,i)_n$ of two MRS quadratic functions has period 
  $2\cdot\mathrm{lcm}(i+j,i-j)$. 
\end{theorem}
\begin{proof}
  The proof of \cite[Theorem 5.2]{cc-bal} explains how the period is to be obtained: it is $2^tm$, where $m$ is the least common multiple of the (automatically odd) multiplicative orders of the roots of $A(x)$ and $2^t$ is the smallest power of $2$ dominating the multiplicities of all of those roots. In that proof, the lines from below equation (5-4) in the proof of Theorem 5.2 to equation (5-6) should be replaced by the lines below, down to the second new equation \eqref{Kprop2}.
  
    In fact, the {\it smallest} 
    \begin{equation*}
      K:=2^t m,\quad m\text{ odd}
    \end{equation*}
    for which $x^K-1$ is divisible by $A(x)$ is obtained for
    \begin{equation*}
      m=\mathrm{lcm}\left(\text{multiplicative orders of the roots of $A(x)$ in the algebraic closure }\overline{\bF_2}\right)
    \end{equation*}
    and the smallest $t$ such that $2^t$ dominates the multiplicity of every root of $A(x)$.

    We write `$\sim$' to indicate that Laurent polynomials are equal up to invertible scaling (e.g. by $x^{\pm 1}$) . Note that $K$ is a period for the sequence $\{v(n)\}$, since
    \begin{align*}
      x^K &= 1\text{ in }\bF_2[x^{\pm 1}]/(A(x))\\
          &\Longrightarrow \mathrm{gcd}\left(x^n-1,\ A(x)\right) \sim \mathrm{gcd}\left(x^{n+tK}-1,\ A(x)\right)\\
          &\Longrightarrow v(n) = v(n+tK),\ \forall t\in \mathbb{Z}.
    \end{align*}
    We see that
    \begin{equation}\label{Kprop2}
      K \text{ is the {\it smallest} period for the sequence } \{v(n)\},
    \end{equation}
    since it was chosen so that 
    \begin{equation}\label{Kprop1}
      A(x) \text{ divides } x^n-1 \text{ if and only if } x^K-1 \text{ does, if and only if } K \text{ divides } n.
    \end{equation}

  The polynomial $A(x)$ is
  \begin{align*}
    x^j+x^{-j} + x^i + x^{-i}&\sim x^{2i}+x^{i+j}+x^{i-j}+1\\
                             &=(x^{i+j}+1)(x^{i-j}+1)\numberthis\label{eq:xii}
  \end{align*}
  (with `$\sim$' as in the proof of \cite[Theorem 5.2]{cc-bal}, meaning `equal up to invertible scaling'). Consider the unique decompositions
  \begin{equation*}
    i+j = 2^a\alpha,\quad i-j = 2^b\beta,\quad \alpha,\beta\text{ odd}.
  \end{equation*}
  The roots of $x^{i+j}+1$ and $x^{i-j}+1$ have orders dividing $\alpha$ and $\beta$ respectively, so the smallest order that will accommodate both is $\mathrm{lcm}(\alpha,\beta)$. As for multiplicities in \Cref{eq:xii}, at the very most, a root might be common to both polynomials, in which case its multiplicity will be $2^a+2^b$. The bound is of course achieved, for the common root 1.

  It follows that
  \begin{align*}
    2^t =\text{smallest power of 2 dominating }2^a+2^b = 2^{\max(a,b)+1},
  \end{align*}
  whence
  \begin{equation*}
    2^tm = 2^{\max(a,b)+1}\mathrm{lcm}(\alpha,\beta) = 2\cdot 2^{\max(a,b)}\mathrm{lcm}(\alpha,\beta) = 2\cdot \mathrm{lcm}(i+j,i-j). 
  \end{equation*}
  This concludes the proof. 
\end{proof}

Finally, to return to \cite[Example 5.1]{cc-bal}:

\begin{corollary}\label{cor:viper}
  The period of the sequence $V(g_i)$ of \cite[Example 5.1]{cc-bal} is
  \begin{equation*}
    \begin{cases}
      2(i+1)(i-1)\quad &\text{if $i$ is even}\\
      (i+1)(i-1)\quad &\text{if $i$ is odd}.
    \end{cases}
  \end{equation*}
\end{corollary}
\begin{proof}
  This is a direct consequence of \Cref{th:gij}: when $i$ is even we have
  \begin{equation*}
    \mathrm{gcd}(i+1,i-1)=1\Longrightarrow\mathrm{lcm}(i+1,i-1)=(i+1)(i-1),
  \end{equation*}
  while for odd $i$ we have
  \begin{equation*}
    \mathrm{gcd}(i+1,i-1)=2\Longrightarrow\mathrm{lcm}(i+1,i-1)=\frac{(i+1)(i-1)}2.
  \end{equation*}
  Now simply double these quantities, as \Cref{th:gij} indicates.
\end{proof}

%%%%%%%%%%%%%%%%%%%%%%%%%%%%%%%%%%%%%%%%%%%%%%%%%%%%%%%%%%%%%%%%%%%%%%%%%%%%%
\subsection{Balancing}\label{subse:bal}

We will now revisit the issue of whether or not quadratic rotation symmetric (RS for short) Boolean functions are balanced, focusing mostly, in the sequel, on {\it binomial} functions: those of the form $(0,i)+(0,j)$.

Quadratic RS functions
\begin{equation}\label{eq:qsum0i}
  Q:=\sum_{i\in I} (0,i).
\end{equation}
split into two dichotomous groups: those that are, on occasion, balanced (over {\it some} $n$), and those that are not. Some language that reflects this:

\begin{definition}\label{def:balrefr}
  A quadratic RS function \Cref{eq:qsum0i} is
  \begin{itemize}
  \item {\it balance-refractory} if it is not balanced for any $n$;
  \item and {\it balance-friendly} if it {\it is} occasionally balanced.
  \end{itemize}
  By \cite[Theorem 5.3]{cc-bal}, if $Q$ has an even number of terms and is balanced friendly, then it is necessarily balanced precisely for
  \begin{equation*}
    n=2^{\nu}(\mathrm{mod}~2^{\nu+1})
  \end{equation*}
  for some $\nu$. 
\end{definition}

With this in place, we will prove

\begin{theorem}\label{th:whenbinombal}
  For positive integers $j<i$, the RS function
  \begin{equation*}
    Q:=(0,j)+(0,i)
  \end{equation*}
  is balance-friendly if and only if $i$ and $j$ have distinct 2-adic valuations.

  In that case, if $\nu:=\min(\nu_2(i),\nu_2(j))$, $Q$ is balanced precisely over 
  \begin{equation*}
    n=2^{\nu+1}(\mathrm{mod}~2^{\nu+2}).
  \end{equation*}  
\end{theorem}

Before proceeding, note the following simple devices for reducing away some of the divisors of the $i$ involved in $(0,i)$:

\begin{lemma}\label{le:redgcd}
  Let $d$ and $n$ be two positive integers.
  \begin{enumerate}[(1)]
  \item\label{item:1} A quadratic RS function $\sum_i (0,di)$ is balanced over $dn$ if and only if $\sum_i (0,i)$ is balanced over $n$.
  \item For $k$ coprime to $n$,
    \begin{equation*}
      \sum_i (0,i)
      \quad\text{and}\quad
      \sum_i (0,ki)
    \end{equation*}
    are simultaneously (un)balanced over $n$.
  \item\label{item:2} In particular, if some $i_0$ is coprime to $n$ and
    \begin{equation*}
      i_0 i' = i (\mathrm{mod}~n),\ \forall i,
    \end{equation*}
    the functions 
    \begin{equation*}
      \sum_i(0,i)
      \quad\text{and}\quad
      (0,1)+\sum_{i\ne i_0}(0,i')
    \end{equation*}
    are simultaneously $n$-(un)balanced.  \qedhere
  \end{enumerate}
\end{lemma}

The last remark, in particular, allows us to always reduce the problem to the case when one of the terms is $(0,1)$; this is occasionally convenient.

\begin{proposition}\label{pr:uniqv}
  Let \Cref{eq:qsum0i} be a quadratic RS function with an even number of terms.
  
  If the smallest 2-adic valuation $\nu$ over $I$ is achieved at a single $i_0\in I$, then $Q$ is balanced precisely over
  \begin{equation*}
    n=2^{\nu+1}(\mathrm{mod}~2^{\nu+2}).
  \end{equation*}
  In particular, $Q$ is balance-friendly.
\end{proposition}
\begin{proof}
  By \cite[Theorem 5.9]{cc-bal} it is enough to study balancing over powers of 2, while \cite[proof of Theorem 5.32]{cc-bal} shows that the {\it only} power of 2 over which $Q$ could be balanced is $2^{d_Q}$, the smallest one dominating valuation
  \begin{equation*}
    d_Q:=\nu_{(x-1)} \sum_{i\in I}\left(x^i-x^{-i}\right)\text{ in }\bF_2[x^{\pm 1}]. 
  \end{equation*}  
  Under our present assumption that
  \begin{equation*}
    \nu:=\nu_2(i_0)\text{ is the unique minimum among }\nu_2(i),\ i\in I,
  \end{equation*}
  \cite[Corollary 5.27]{cc-bal} shows that $d_Q$ is in fact the $(x-1)$-valuation of the single summand
  \begin{equation*}
    x^{i_0}-x^{-i_0}\sim x^{2i_0}-1,
  \end{equation*}
  which in turn is easily seen to be precisely $\nu+1$. In short, the question of whether or not $Q$ is ever balanced (over any $n$) reduces to $n=2^{\nu+1}$. Now, because all $i\in I$, along with $n=2^{\nu+1}$, are all divisible by $2^\nu$, \Cref{le:redgcd} \Cref{item:1} further boils down the question to $n=2$ with $i_0$ the unique odd $i\in I$.
  
  The goal, by \cite[Theorem 2.1 and Proposition 5.21]{cc-bal}, is to show that the trace function
  \begin{equation}\label{eq:q'}
    Q'(x):=\mathrm{Tr}_{\bF_4/\bF_2}\left(\sum_{i\in I} x^{2^i+1}\right),\ x\in \bF_{2^2}
  \end{equation}
  does {\it not} vanish identically on the field $\bF_4$ with 4 elements. To that end, note first that all $i\in I$ {\it different} from $i_0$ are assumed divisible by 2, whence all $2^i$, $i\in i_0$ appearing among the exponents of \Cref{eq:q'} are powers of 4, hence
  \begin{equation*}
    x^{2^i} = x\Longrightarrow x^{2^i+1} = x^2,\ \forall i\ne i_0\in I\text{ and }x\in \bF_4. 
  \end{equation*}
  Since there is an {\it odd} number of terms indexed by $i\ne i_0$, \Cref{eq:q'} equals
  \begin{equation}\label{eq:tr2terms}
    \mathrm{Tr}\left(x^{2^{i_0}+1} + x^2\right). 
  \end{equation}
  because $i_0$ is odd however,
  \begin{equation*}
    x^{2^{i_0}} = x^2\Longrightarrow x^{2^{i_0}+1} = x^3,\ \forall x\in \bF_4. 
  \end{equation*}
  All cubes in $\bF_4$ are elements of $\bF_2\subset \bF_4$, so their trace (when regarded as elements of $\bF_4$) vanishes. For that reason, the first term in \Cref{eq:tr2terms} simply drops and $Q'(x)$ equals $\mathrm{Tr}(x^2)$. The trace itself does not vanish and $x\mapsto x^2$ is an automorphism, so we are done. 
\end{proof}

In particular, specializing to only {\it two} terms:

\begin{corollary}\label{cor:binom-uniqv}
  Suppose the two positive integers $j<i$ have different 2-adic valuations. The RS function $(0,j)+(0,i)$ is balance-friendly, and balanced precisely over
  \begin{equation*}
    n=2^{\nu+1}(\mathrm{mod}~2^{\nu+2}),
  \end{equation*}
  where $\nu:=\min\left(\nu_2(j),\ \nu_2(i)\right)$.  \qedhere
\end{corollary}

\pf{th:whenbinombal}
\begin{th:whenbinombal}
  \Cref{cor:binom-uniqv} disposes of half of the claim, leaving only the case $\nu_2(i)=\nu_2(j)$ to contend with (the goal being to prove that $Q$ is never balanced). The problem undergoes a number of simplifications:
\begin{itemize}
 
\item By \Cref{cor:binom-uniqv}, we can assume that $\nu_2(i)$ and $\nu_2(j)$ are equal:
\begin{equation*}
  i=2^{\mu}i',\ j=2^{\mu}j'
  \quad\text{with}\quad
  i',\ j'\text{ odd}. 
\end{equation*}

\item It follows, then, that the $(x-1)$-adic valuation of
  \begin{equation*}
    A(x)=A_Q(x)
    :=
    x^{i}+x^{-i}+x^j+x^{-j}\sim (x^{i+j}+1)(x^{i-j}+1)
  \end{equation*}
  is
  \begin{equation}\label{eq:dq}
    d=d_Q=2^{\nu_2(i+j)} + 2^{\nu_2(i-j)}.
  \end{equation}

\item Since these two latter valuations are different (exactly one of $i+j$ and $i-j$ is divisible by $2^{\mu+2}$), the smallest power of 2 dominating $d_Q$ is
  \begin{equation}\label{eq:nq}
    n=n_Q:=2^{\max\left(\nu_2(i+j),\nu_2(i-j)\right) + 1}.
  \end{equation}

\item The proof of \cite[Theorem 5.32, case 1]{cc-bal} shows that whether or not $Q$ is ever balanced (over any $n$ at all) boils down to the same question over $n_Q$.

\item Since $n_Q$ is divisible by $2^{\mu}$ just as $i$ and $j$ are, \Cref{le:redgcd} \Cref{item:1} allows us to divide through by that common factor and assume $i$ and $j$ are odd.

\item And furthermore, $i$ and $j$ being odd now and hence coprime to \Cref{eq:nq}, \Cref{le:redgcd} \Cref{item:2} further allows us to assume, when convenient, that $j<i$ is in fact $1$.
\end{itemize}

Consider, then, a RS function $(0,1)+(0,i)$, and its associated trace function
\begin{equation}\label{eq:q'i1}
  Q':=\mathrm{Tr}\left(x^3+x^{2^i+1}\right).
\end{equation}
$Q$ is balanced over $n$ if and only if $Q'$ is balanced on $\bF_{2^n}$ \cite[Theorem 2.1]{cc-bal}, and this in turn depends only on the behavior at the ``critical'' value $n=n_Q=2^{\nu+1}$ given by \Cref{eq:nq}, with
\begin{equation*}
  j=1,\ \nu:=\max(\nu_2(i-1),\nu_2(i+1)). 
\end{equation*}

$Q'$ is {\it un}balanced on $\bF_n$ if and only if it vanishes identically on the $\bF_2$-subspace
\begin{equation}\label{eq:kerf1}
  \ker\left(F-\id\right)^{d_Q}\subset \bF_{2^{2^{\nu+1}}}
\end{equation}
\cite[discussion preceding Proposition 5.18]{cc-bal}, where
\begin{itemize}
\item $F$ is the Frobenius automorphism $x\mapsto x^2$ on the algebraic closure $\overline{\bF_2}$;
\item and $d_Q$ is the $(x-1)$-adic valuation of $A(x)$, as in \Cref{eq:dq} with $j=1$:
  \begin{equation}\label{eq:dspec1i}
    d_Q=2^{\nu_2(i+1)} + 2^{\nu_2(i-1)} = 2 + 2^{\nu}. 
  \end{equation}
\end{itemize}
Now \Cref{eq:dspec1i} and
\begin{equation*}
  \ker\left(F-\id\right)^{2^{\nu}} = \ker\left(F^{2^{\nu}}-\id\right) = \bF_{2^{2^{\nu}}}\subset \bF_{2^{2^{\nu+1}}}
\end{equation*}
jointly imply that \Cref{eq:kerf1} equals
\begin{equation*}
  \{x\in \bF_{2^{2^{\nu+1}}}\ |\ (F^2-\id)x = x^4-x\in \bF_{2^{2^{\nu}}}\}.
\end{equation*}
An element $x\in \bF_{2^{2^{\nu+1}}}$ such that
\begin{equation*}
  x^4-x+\beta=0\text
  \quad\text{for some }\beta\in \bF_{2^{2^{\nu}}}
\end{equation*}
has degree $\le 2$ over $\bF_{2^{2^\nu}}$ (because $[\bF_{2^{2^{\nu+1}}}:\bF_{2^{2^\nu}}]=2$), so there is a factorization
\begin{equation*}
  x^4-x+\beta = \left(x^2-x+\alpha\right)\left(x^2-x+\alpha+1\right)
  \quad\text{for some}\quad
  \alpha\in \bF_{2^{2^{\nu}}}. 
\end{equation*}
We will henceforth focus on roots $x$ of a single factor, say $x^2-x+\alpha$. There are two cases to consider:
\begin{enumerate}[(a)]
\item {\bf The trace $\mathrm{Tr}_{\bF_{2^{2^{\nu}}}/\bF_2}(\alpha)$ vanishes.} This is just another way of saying that the roots of $x^2-x+\alpha$ already belong to $\bF_{2^{2^{\nu}}}$. But then so do $x^3$ and $x^{2^i+1}$, so their $\bF_2$-valued traces on the order-two extension $\bF_{2^{2^{\nu+1}}}$ of $\bF_{2^{2^{\nu}}}$ vanish.

  In short, in this case \Cref{eq:q'i1} vanishes because both of its terms do.
  
\item {\bf The trace $\mathrm{Tr}_{\bF_{2^{2^{\nu}}}/\bF_2}(\alpha)$ does not vanish.} Or: $x^2-x+\alpha$ is irreducible over $\bF_{2^{2^{\nu}}}$, and hence any of its roots generates the extension $\bF_{2^{2^{\nu+1}}}/\bF_{2^{2^{\nu}}}$.

  We relegate to \Cref{le:alpharec} the fact that for any positive integer $k$ we have
  \begin{equation}\label{eq:alphasum}
    \mathrm{Tr}_{\bF_{2^{2^{\nu+1}}}/\bF_{2^{2^{\nu}}}}x^{2^k+1}
    =
    1+\alpha+\alpha^2+\alpha^4+\cdots+\alpha^{2^{k-1}}.
  \end{equation}
  For odd $k$ (such as $k=1$ or $k=i$) the number of $\alpha^{2^u}$ terms in the above sum is odd, and all of those terms have the same $\left(\bF_{2^{2^{\nu}}}/\bF_2\right)$-trace (namely 1). On the other hand $\mathrm{Tr}_{\bF_{2^{2^{\nu}}}/\bF_2}(1)=0$, so that \Cref{eq:alphasum} equals 1 for both $k=1$ and $k=i$.

  Once more, then, \Cref{eq:q'i1} vanishes (this time because both terms are equal to 1).
\end{enumerate}
This finishes the proof of the theorem. 
\end{th:whenbinombal}

\begin{lemma}\label{le:alpharec}
  Let $u$ and $v$ be the two roots of the polynomial $x^2-x+\alpha$ over a field $\bF$ of characteristic 2.
  
  For any $k\in \bZ_{\ge 0}$ we have
  \begin{equation*}
    u^{2^k+1}+v^{2^k+1} = 1+\alpha+\alpha^2+\alpha^4+\cdots+\alpha^{2^{k-1}}.
  \end{equation*}
\end{lemma}
\begin{proof}
  The sequence $S_s:=u^s+v_s$ satisfies
  \begin{equation*}
    S_0=0,\ S_1=1\ S_{n+2} = =S_{n+1}+\alpha S_n,\ \forall n\ge 0.
  \end{equation*}
  The recursion can be written in matrix form as
  \begin{equation*}
    \begin{pmatrix}
      S_{n+2}\\
      S_{n+1}
    \end{pmatrix}
    =
    \begin{pmatrix}
      1&\alpha\\
      1&0
    \end{pmatrix}
    \begin{pmatrix}
      S_{n+1}\\
      S_{n}
    \end{pmatrix},
  \end{equation*}
  whence
  \begin{equation*}
    \begin{pmatrix}
      S_{n+1}\\
      S_n
    \end{pmatrix}
    =
    \begin{pmatrix}
      1&\alpha\\
      1&0
    \end{pmatrix}^n
    \begin{pmatrix}
      S_1\\
      S_0
    \end{pmatrix}
    =
    \text{first column of }
    \begin{pmatrix}
      1&\alpha\\
      1&0
    \end{pmatrix}^n.
  \end{equation*}
  Now it remains to observe that
  \begin{equation*}
    \begin{pmatrix}
      1&\alpha\\
      1&0
    \end{pmatrix}^{2^k}
    =
    \begin{pmatrix}
      1+\alpha_k & \alpha\\
      1&\alpha_k
    \end{pmatrix}
    \quad\text{with}\quad
    \alpha_k:=\alpha+\alpha^2+\alpha^4+\cdots+\alpha^{2^{k-1}},
  \end{equation*}
  admitting a simple inductive proof by repeated squaring. 
\end{proof}

A word of caution (retaining the notation in the proof of \Cref{th:whenbinombal}): although for powers $2^{\mu}$, $\mu\le \nu$, $Q'$ vanishes identically on $\bF_{2^{2^{\mu}}}$ \cite[Proposition 5.21]{cc-bal}, $Q'$ need not vanish identically on the entire critical field $\bF_{2^{2^{\nu+1}}}$.

\begin{example}\label{ex:13}
  Consider the RS function $Q:=(0,1)+(0,3)$. The critical value of $n$ for checking balancing is \Cref{eq:nq}, i.e.
  \begin{equation*}
    2^{\max(\nu_2(4),\nu_2(2))+1} = 2^3=8,
  \end{equation*}
  so whether or not $Q$ is balanced amounts to whether or not
  \begin{equation}\label{eq:q'13}
    Q'=\mathrm{Tr}\left(x^3+x^9\right)
  \end{equation}
  is balanced on $\bF_{2^8}$; that this is not so follows from \Cref{th:whenbinombal}. On the other hand, because the $d_Q$ of \Cref{eq:dq} is in this case $4+2=6$, \Cref{eq:q'13} vanishes identically on
  \begin{equation*}
    \ker\left(F-\id\right)^6\subset \bF_{2^8}. 
  \end{equation*}
  On the other hand, $Q'$ does not vanish on all of $\bF_{2^8}$.

  To check the latter claim, notice first that the non-zero cubes $x^3$, $x\in \bF_{2^8}$ are precisely the roots of unity whose order divides
  \begin{equation*}
    \frac{2^8-1}{3} = (2^2+1)(2^4+1) = 5\cdot 17.
  \end{equation*}
  For that reason, whether or not $Q'\equiv 0$ amounts to whether every $(5\cdot 17)^{th}$ root of unity $y$ in $\bF_{2^8}$ (such as $y=x^3$) has the same trace as its cube (i.e. $x^9=(x^3)^3$).

  Now, take for $y$ a primitive $17^{th}$ root of unity. Because the multiplicative order of 2 modulo 17 is 8, the cyclotomic polynomial
  \begin{equation*}
    \Phi_{17}(x) = x^{16}+\cdots + x + 1
  \end{equation*}
  factors over $\bF_2$ as a product $\Phi_{17}=f\cdot g$ of two polynomials of order 8 \cite[Theorem 2.47]{MR1429394}. Because the first sub-dominant coefficient in $\Phi_{17}$ (i.e. the coefficient of $x^{16-1}=x^{15}$) is non-zero, the sub-dominant coefficients of $f$ and $g$ differ. In other words: the traces (over $\bF_2$) of the roots of $f$ differ from those of $g$.

  Now simply note that $3$ is not a quadratic residue modulo $17$ (either by a direct check or quadratic reciprocity \cite[\S 5.2, Theorem 1]{ir-nt}), so must generate the order-16 cyclic multiplicative group $(\bZ/17)^{\times}$. It follows that cubing cycles through all $17^{th}$ roots of unity in $\bF_{2^8}$, so some for some root $y$ of $f$ its cube $y^3$ is a root of $g$. Consequently, there are cubes $x^3=y$ whose own cubes $x^9=y^3$ have different trace, and hence 
  \begin{itemize}
  \item $\mathrm{Tr}_{\bF_8/\bF_2}\left(x^3+x^9\right)$ vanishes on $\ker\left(F-\id\right)^6\subset \bF_8$;
  \item but not on $\ker\left(F-\id\right)^8  = \bF_8$.
  \end{itemize}
\end{example}

%%%%%%%%%%%%%%%%%%%%%%%%%%%%%%%%%%%%%%%%%%%%%%%%%%%%%%%%%%%%%%%%%%%%%%%%%%%%%
\subsection{Unbalanced functions}\label{subse:unbal}

Consider a quadratic RS function
\begin{equation}\label{eq:qgen}
  Q = \sum_{i\in I}(0,i). 
\end{equation} 
We know that its weight at $n$ is of the form
\begin{equation} \label{eq:pmW}
  wt(Q_n) = 2^{n-1}\pm 2^{\frac{n+v(n)}2-1}\text{ or }2^{n-1},
\end{equation}
where $v(n)=v_Q(n)$ is the plateau parameter associated to $Q$ (see 
\Cref{quadwt} and the discussion of plateaued functions in \Cref{se:int}).

When the function is {\it not} balanced, then, knowing its weight amounts to knowing which of the two signs obtains; the following terms are meant to convey this:

\begin{definition}\label{def:overunder}
  A Boolean function in $n$ variables is
  \begin{itemize}
  \item {\it overbalanced} if its weight is strictly larger than $2^{n-1}$;
  \item and {\it underbalanced} if the weight is strictly smaller than $2^{n-1}$.
  \end{itemize}
\end{definition}

It is known \cite[p. 172]{carl} that determining the $\pm$ sign attached to the Walsh transform value in \Cref{eq:pmW} is often much more difficult than determining the Walsh value itself. The results below in this section shed a little light on this issue.

In studying balancing, the trace function
\begin{equation}\label{eq:q'gen}
  Q'(x):=\sum_{i\in I}\mathrm{Tr}~x^{2^i+1}
\end{equation}
associated to $Q$ is of great help, as seen above: $Q$ is balanced at $n$ if and only if $Q'$ is balanced over $\bF_{2^n}$. Note, however, that over or underbalancing are not transferable in the same fashion:

\begin{lemma}\label{le:qq'overunder}
  For any quadratic RS function \Cref{eq:qgen} with associated trace function \Cref{eq:q'gen} there is an $n$ at which $Q$ is underbalanced while $Q'$ is overbalanced.
\end{lemma}
\begin{proof}
  By \cite[Theorem 5.2]{cc-rec} (and its proof), the weights of $Q$ and $Q'$ at $n$ are
  \begin{equation}\label{eq:alphagamma}
    2^{n-1}-\frac{\sum\alpha_i^n}2
    \quad\text{and}\quad
    2^{n-1}+\frac{\sum\gamma_j^n}2
  \end{equation}
  respectively, where $\{\alpha_i\}$ and $\{\gamma_j\}$ are equinumerous multisets of roots of unity scaled by $\sqrt 2$. If $n$ is divisible by the orders of all $\frac{\alpha_i}{\sqrt 2}$ and $\frac{\gamma_j}{\sqrt 2}$ then $Q$ will be underbalanced, while $Q'$ will be overbalanced. 
\end{proof}

It is also easy enough to see, in concrete examples, how it comes about that Boolean and trace functions diverge in their over/underbalancing behavior. 

\begin{example}\label{ex:monbooltr1}
  Consider the monomial RS function $Q=(0,1)$. According to \cite[Theorem 8]{kph} $Q$ is balanced precisely over odd $n$, and underbalanced elsewhere; in particular, it will be underbalanced at $n=4$.
  
  On the other hand, the corresponding trace function $Q':=\mathrm{Tr}~x^3$ is easily seen to be {\it over}balanced at $n=4$: all non-zero elements of $\bF_{2^4}$ are $15^{th}$ roots of unity, so the cubes cubes $x^3$ are 0 and the $5^{th}$ roots of unity. Because the multiplicative order of 2 modulo 5 is 4, the cyclotomic polynomial
  \begin{equation*}
    \Phi_5(x) = x^4+x^3+x^2+x+1
  \end{equation*}
  stays irreducible over $\bF_2$ \cite[Theorem 2.47]{MR1429394}, and hence each of its roots generates the extension $\bF_{2^4}/\bF_2$ and they all have trace 1. The cubes $x^3$ with vanishing trace, then, are $0$ and $1$; for that reason,
  \begin{equation*}
    \ker\left(x\mapsto \mathrm{Tr}~x^3\right) = \{0\}\cup\{x\in \bF_{2^4}\ |\ x^3=1\}. 
  \end{equation*}
  This set has cardinality 4, so the weight of $Q'$ over $\bF_{2^4}$ is $2^4-4 = 12 > 2^3$.
\end{example}

\begin{remark}
  \Cref{ex:monbooltr1} also verifies the fact that whenever $Q$ and $Q'$ have different weights (at $n$, say), those weights must be ``complementary'' in the sense that they add up to $2^{n-1}$: in this specific case, for $n=4$ we have 
  \begin{equation}\label{eq:wtqn4}
    wt(Q) = 2^{n-1} - 2^{\frac{n+v}2-1} = 2^3 - 2^{\frac{4+v}2-1}
  \end{equation}
  for the plateau parameter $v$ of $Q$ at $4$ (a consequence of \cite[Theorem 8]{kph} again, and the fact that quadratic Boolean functions are plateaued). That plateau parameter equals
  \begin{equation*}
    \deg\mathrm{gcd}\left(x^4-1,x^2-1\right) = 2
  \end{equation*}
  \cite[Lemma 3.13]{cc-bal}, so that \Cref{eq:wtqn4} equals $2^3-2^2$; compare this to the weight $2^3+2^2$ of $Q'$ in \Cref{ex:monbooltr1}. 
\end{remark}

Not only are $Q$ and $Q'$ not, in general, simultaneously over(under)balanced, but the over/under pattern can differ qualitatively. This is already visible with $Q=(0,1)$ and $Q'=\mathrm{Tr}~x^3$: as observed, $Q$ is either balanced or underbalanced, whereas $Q'$ can be both overbalanced (\Cref{le:qq'overunder}) and {\it under}balanced: $\mathrm{Tr}~x^3$ vanishes, for instance, on $\bF_{2^2}=\bF_4$. Underbalancing also occurs at higher, ``unproblematic'' $n\ge 3=2\cdot 1+1$:

\begin{example}\label{ex:q'undertoo}
  With the same $Q$ and $Q'$ as in \Cref{ex:monbooltr1}, consider the case $n=6$ (so the relevant field to take traces over is $\bF_{2^6}$).

  The non-zero cubes $x^3$, $x\in \bF_{2^6}$ are roots of unity of orders dividing $\frac{2^6-1}{3}=21=3\cdot 7$, giving the following trace tally:
  \begin{itemize}
  \item 0 naturally has trace 0, giving one 0-trace element.
  \item 1 also has trace 0, giving three more $x$ with $x^3=1\Rightarrow \mathrm{Tr}~x^3=0$.
  \item the two primitive order-3 roots have trace 1 in their splitting field $\bF_{2^2}$, over which $\bF_{2^6}$ has odd degree. This gives $6=2\cdot 3$ elements
    \begin{equation*}
      x^3=\text{primitive order-3 root}\Longrightarrow \mathrm{Tr}~x^3=1.
    \end{equation*}
  \item the six primitive order-7 roots belong to the smaller field $\bF_{2^3}$ over which $\bF_{2^6}$ has {\it even} degree; their traces thus vanish, and we have $18=6\cdot 3$ elements
    \begin{equation*}
      x^3=\text{primitive order-7 root}\Longrightarrow \mathrm{Tr}~x^3=0.
    \end{equation*}
  \item As for the primitive roots of order 21, they are roots of the cyclotomic polynomial $\Phi_{21}$ of degree 12, which factors over $\bF_2$ as a product of two polynomials $f$ and $g$ of degree 6 \cite[Theorem 2.47]{MR1429394}. Because the first sub-dominant coefficient of $\Phi_{12}$ is 1, those of $f$ and $g$ must be 0 and 1, meaning that the roots of $f$ and $g$ have different traces. This provides as many elements of trace 0 as of trace 1, so the two cancel.    
  \end{itemize}

  On balance, then, there are more $x$ with $\mathrm{Tr}~x^3=0$ that with trace 1. Specifically, per the above discussion,
  \begin{equation*}
    \sharp\left\{x\in \bF_{2^6}\ |\ \mathrm{Tr}~x^3=0\right\}
    -
    \sharp\left\{x\in \bF_{2^6}\ |\ \mathrm{Tr}~x^3=1\right\}
    =
    18-6+4 = 16.
  \end{equation*}
  This shows that in this case the weights of $Q$ and $Q'$ do coincide, and are equal to $2^5-2^3$.
\end{example}

\begin{remark}\label{re:alphanegamma}
  It is plain from \Cref{ex:monbooltr1,ex:q'undertoo}, for instance, that the $\alpha$s and $\gamma$s of \Cref{eq:alphagamma} cannot coincide as multisets, despite some initial plausibility (they are multisets of equal sizes, consisting of roots of unity scaled by the same amount).
\end{remark}

%%%%%%%%%%%%%%%%%%%%%%%%%%%%%%%%%%%%%%%%%%%%%%%%%%%%%%%%%%%%%%%%%%%%%%%%%%%%%
%%%%%%%%%%%%%%%%%%%%%%%%%%%%%%%%%%%%%%%%%%%%%%%%%%%%%%%%%%%%%%%%%%%%%%%%%%%%%
\section{Sums of more than two functions} \label{se:grt2}

We will provide a method for computing the period of the list of $v$-values for the sum of any number of quadratic monomials $(0,i)_n,$ but it does not seem possible to obtain a concise formula for these periods like the one in \Cref{th:gij}.

Define $Q_I=Q_{I,n} = \sum_{i \in I} (0, i)_n,$ where $I = \{a(i):~0 < i \leq u\}$. It is convenient to use $V(I)$ instead of $V(Q_I)$ for the period length of the $v$-values $v(n) = v(n,I).$
\begin{theorem} \label{th:QI}
The list of $v$-values for $Q_{I,n}$ has period $V(I)$ equal to the smallest positive integer $k$ such that $x^k-1$ is divisible mod $2$ by
$$A_I(x) = x^{2a(1)} + \sum_{j=2}^{u} (x^{a(1)+a(j)} + x^{a(1)-a(j)}) + 1.$$
\end{theorem}
\begin{proof}
We have 
\begin{equation} \label{eq:deg}
2a(1) = \deg(\gcd\nolimits_2(x^k - 1, A_I(x))
\end{equation}
($\gcd_2$ means the gcd is taken mod $2$) for some integers $k,$ where $2a(1)=\max v(n,I)$ is the unique largest value of $v(n,I)$ (by \Cref{th:prd}). It follows from the definition of $A_I(x)$ that the smallest such $k$ is the one specified in the theorem.  
\end{proof}

Since \cite[Theorem 5.2]{cc-bal} applies to the functions $Q_{I,n},$ computing the smallest $k$ in \Cref{th:QI} is the same as computing $K=2^tm$ in the proof of \Cref{th:gij}.  There is a straightforward algorithm to compute $V(I)$ using \Cref{th:QI} and the definition of $K=2^tm.$ By 
\Cref{th:QI}, every irreducible factor $f(x)$ of polynomial $A_I(x)$ has the property that $f(x)^r$ divides $x^k-1,$ where $r$ is the power to which $f(x)$ occurs in the factorization of $A_I(x).$ Given an irreducible factor $f(x)$ of degree $d,$ adjoining any of its roots generates the field $\bF_{2^d}.$ Since this is also the set of roots of $x^{2^d}-x,$ we have $f(x)$ divides $x^{2^d-1}-1.$ Hence the minimal $k$ such that $f(x)$ divides $x^k-1$ has $k$ divides $2^d-1.$ Now factoring $A_I(x)$ mod $2$ and using the divisibility conditions on $k$ given by the irreducible factors, plus computing the value of $t$ in $K=2^tm,$ we can obtain the smallest $k.$ 

We give a few examples to illustrate how easily these computations to determine $K=2^tm$ can be carried out. We used Mathematica for the polynomial mod $2$  $\gcd$ computations. Notice that in the examples we frequently use the well-known fact $$\gcd((x^a - 1),(x^b -1))= x^{gcd(a,b)} - 1$$ without comment. We also mention the values of $d$ such that $m$ divides $2^d-1,$ as explained in the previous paragraph, without comment. It is convenient to save space by defining $$A_I(x) = x^{2a(1), a(1)+a(2), a(1)+a(3), \ldots, a(1)+a(j), a(1)-a(j), a(1)-a(j-1), \ldots, a(1)-a(2),0}.$$

\begin{example}[I=\{7,4,1\}, V(I)=72]  \label{ex:ex1}
We have A(x,I)= $x^{14,11,8,6,3,0} = (x+1)^8(x^6+x^3+1)$ mod $2$ for the factorization into irreducible polynomials, so $t=3$ and $x^m - 1$ divides
$x^{2^6-1}=x^{63}.$  Since $x^9-1$ factors as  $(x+1)(x^2+x+1)(x^6+x^3+1)$ mod $2$ we have $$x^{72}-1 = (x^9-1)^8 = (x+1)^8(x^2+x+1)^8(x^6+x^3+1)^8 ~\text{mod}~ 2. $$ Thus $K = 8 \cdot 9 = 72$ is the least value of $k$ which gives the required degree $14$ in \Cref{eq:deg}.
\end{example} 
\begin{example}[I1=\{5,3,2,1\} and I2=\{5,4,2\}, period $34$]  \label{ex:ex2}
We have A(x,I1)= $x^{10,8,7,6,4,3,2,0} = (x+1)^2(x^8+x^5+x^4+x^3+1)$ mod $2$ for the factorization into irreducible polynomials, so $t=1$ and 
$x^m - 1$ divides $x^{2^8-1}.$  Since $x^{17}-1$ factors as  $(x+1)(x^8+x^5+x^4+x^3+1)(x^8+x^7+x^6+x^4+x^2+x+1)$ mod $2$ we have 
$$x^{34}-1 = (x+1)^2(x^8+x^5+x^4+x^3+1)^2(x^8+x^7+x^6+x^4+x^2+x+1)^2 ~\text{mod}~ 2.$$ Thus  for I1, $K = 34$ is the least value of $k$ which gives the required degree $10$ in \Cref{eq:deg}.  For I2, A(x,I2)= $x^{10,9,7,3,1,0} = (x+1)^2(x^8+x^7+x^6+x^4+x^2+x+1)$ mod $2$ is the product of irreducible polynomials, so again $t=1$ and $x^m-1$ divides $x^{2^8-1}.$ The degree $8$ factor appears in the above factorization of $x^{34}-1,$  so the previous argument for $\text{I}=\text{I}1$ proves that for $\text{I}=\text{I}2$ again $K = 34$ is the least value of $k$ which gives the required degree $10$ in \Cref{eq:deg}.
\end{example}
\begin{example}[I=\{6,2,1\}, V(I)=102]   \label{ex:ex3}
We have A(x,I)= $x^{12,8,7,5,4,0} = (x+1)^2(x^2+x+1)(x^8+x^7+x^6+x^4+x^2+x+1)$ mod $2$ for the factorization into irreducible polynomials, so $t=1,
x^2+x+1$ divides $x^3-1$ and $x^8+x^7+x^6+x^4+x^2+x+1$ divides $x^{255}-1=x^{15 \cdot 17}-1.$ We saw in \Cref{ex:ex2} that the degree $8$ polynomial divides $x^{17}-1$ so a calculation shows that $K=102=2 \cdot 3 \cdot 17$ is the least value of $k$ which gives the required degree $12$
in \Cref{eq:deg}.
\end{example}

Now we can summarize the progress for RS quadratic functions $Q_{I,n}$ on finding the weights, the nonlinearity, the $v(n)$ values and period $V(I),$ the values of $n$ for which the functions are balanced and the affine equivalence classes.  

The results for MRS functions $(0,t)_n$ are essentially complete.  In this case the weights are easily computed from the roots of the characteristic polynomial of the rules matrix (see \Cref{se:int}), without the need to compute initial values for the recursion (this will be an exponential size calculation, and so impossible if the 
recursion order is large). This follows because the {\it Easy Coefficients Conjecture}, stated below, was proved for MRS quadratic functions in \cite[Theorem 4.4]{cc-rec}. Let $M(I) = \max_{i \in I} i.$ Note that from \cite{cus_rec} the rules matrix for 
$Q_{I,n}$ always has size $2^{M(I)}+1$ and the characteristic polynomial always has a single root $2,$ with the remaining roots being irrational complex numbers with various multiplicities.

\begin{conjecture} {\bf Easy Coefficients Conjecture}  \label{ECC}
The recursion for the weights of $Q_{I,n}$ extended backwards from 
$n = 2 M(I) +1$ (to avoid bent functions) to $n=1,$ generates a sequence $w_1, w_2, \ldots$ (with $w_n = wt(Q_{I,n})$ for 
$n \geq 2 M(I) +1$) such that
\begin{equation} \label{eq:27}
w_n=2^{n-1} - \frac{1}{2}(\delta_1^n + \ldots + \delta_{2^{M(I)}}^n),~n = 1, 2, \ldots .
\end{equation}
Here $\delta_1, \ldots, \delta_{2^{M(I)}}$ is the list of the $2^{M(I)}$ irrational roots of the characteristic polynomial for the rules matrix of $Q_{I,n},$ with the distinct roots $\delta_1, \ldots, \delta_{m(Q(I))}$ of the minimal polynomial for $Q_{I,n}$ listed first. The remaining roots are various duplicates of the first $m(Q(I))$ roots.
\end{conjecture}

For $(0,t)_n,$, we know \cite[Theorem 3.4]{cc-bal} the minimal polynomial for the rules matrix is
$$x^{2t+1}-2 x^{2t}-2^t x+2^{t+1} = (x-2)(x^{2t}-2^t)$$
and the remaining roots of the characteristic polynomial, of degree $2^t+1$, are various duplicates of the $2t$ irrational roots.

In the MRS quadratic case $(0,t)_n$, the weights and nonlinearities were already determined in $2009$ \cite[Theorem 8]{kph} by a different method. That method does not give an explicit description of those $n$ for which $(0,t)_n$ is balanced, but, as remarked in \Cref{se:int}, \cite[Lemma 4.2]{cc-bal} does. The answer is that for $n \geq 2t+1$ $(0,t)_n$ is balanced except when 
$n \equiv 0 \bmod{2^{\nu_2(t)+1}}.$ In \cite[Theorem 3.17]{cc-bal}, it is shown that $v(n) =\gcd(n,2t)$ if $(0,t)_n$ is balanced and $v(n) = 2\gcd(n,2t)$ otherwise, and \cite[Lemma 4.1]{cc-bal} gives a complete description of the $v(n)$ sequence for 
$(0,t)_n$, including a proof that the period length and maximum value are both $2t.$ The count $\tau(n) -1$ ($\tau(n)$ denotes the number of positive integer divisors of $n$) for the number of affine equivalence classes for the functions $(0,t)_n,~n \geq 3,$ is proved in \cite[Theorem 4.3]{cc-bal}.

For binomial quadratic RS functions, if \Cref{ECC} (call it the ECC for short) is true the above remarks about ease of computation of the weights will apply.  We believe that the ECC is true for all quadratic RS functions and indeed that a version of the ECC is probably true for RS functions of any degree.

With respect to the $v$-values, \Cref{th:gij}  gives a nice formula for the length of their period for the binomial quadratic RS functions, but for sums of more functions we have nothing beyond \Cref{th:prd} and the method for finding $V(I)$ described below \Cref{th:QI}. \Cref{th:whenbinombal} gives a precise answer to the question of when binomial quadratic RS functions are balanced. This completes, for the binomial case, the work begun in \cite[Theorem 5.5]{cc-bal}. \Cref{se:binrules} gives a complete description of the rules matrices for the binomial case, which can be used to
to do detailed computations for any information about the numbers $v(n)$ and the nonlinearity.

\section{Scaled roots of unity}\label{se:scaled1}

The computational evidence on the rules matrices for $(0,j)+(0,i)$ suggests that the eigenvalues (except for the single 2) are roots of unity scaled by $\sqrt 2$. This was also the case for {\it monomial} functions, whence the relevance of \cite[Proposition 2.8]{cc-rec} on how the numbers $\sqrt 2\cdot \zeta$, for roots of unity $\zeta$, coalesce into Galois conjugacy classes. As a side-note, we expand that remark to algebraic integers of the form $\sqrt d~\zeta$ for square-free $d\in \bZ$ and roots of unity $\zeta$.

Suppose, now that $\ord(\zeta)=n$, i.e. $\zeta$ is a primitive $n^{th}$ root of unity. There are two possibilities to consider:
\begin{itemize}
\item When $n$ is divisible by 4 the cyclotomic polynomial $\Phi_n$ is a polynomial in $x^2$ \cite[\S 9.1, Exercise 12]{cox-gal}. In that case
  \begin{equation}\label{eq:n04}
    \left\{\sqrt{d}~ \zeta\ |\ \ord(\zeta)=n\right\}
  \end{equation}
  are the roots of an integral monic polynomial, so decomposes as a disjoint union of Galois conjugacy classes. 
\item When $n$ is either odd or $2(\mathrm{mod}~4)$ write $n_o$ for the largest odd divisor of $n$, so that $n$ is either $n_o$ or $2n_o$ and this time around
  \begin{equation*}
    \Phi_{n_0}(x^2) = \Phi_{n_0}(x)\Phi_{n_0}(-x) = \Phi_{n_0}(x)\Phi_{2n_0}(x) = \Phi_{2n_0}(-x^2)
  \end{equation*}
  (as follows from \cite[\S 9.1, Exercise 12]{cox-gal} again). The set
  \begin{equation}\label{eq:n24}
    \left\{\sqrt{d}~ \zeta\ |\ \ord(\zeta)=n_o\text{ or }2n_o\right\}
  \end{equation}
  can be partitioned into Galois classes (in place of \Cref{eq:n04}). 
\end{itemize}

Consider, now, the modification of the cyclotomic polynomial $\Phi_n$ that will appropriately accommodate \Cref{eq:n04} or \Cref{eq:n24} as its roots:

\begin{definition}\label{def:modphi}
  For a square-free integer $d$ and a positive integer $n$ with Euler function $\varphi(n)=\deg\Phi_n$ we set
  \begin{equation*}
    \widetilde{\Phi}_{n,d}(x):=
    \begin{cases}
      d^{\frac{\varphi(n)}2}\Phi_n\left(\frac x{\sqrt d}\right) &\text{ if }4|n\\
      d^{\varphi(n)}\Phi_{n_o}\left(\frac {x^2}{d}\right) &\text{ otherwise },\\
    \end{cases}
  \end{equation*}
  where as before, $n_o$ is the largest odd divisor of $n$.

  By construction, $\widetilde{\Phi}_{n,d}$ is monic.
\end{definition}

Unqualified polynomial (ir)reducibility refers to $\bZ[x]$ or, equivalently in all cases of interest, $\bQ[x]$; whenever we consider the issue over other coefficient fields we will make this explicit.

\begin{theorem}\label{th:whenmodphired}
  Let $d$ be a square-free integer and $n\in \bZ_{>0}$. The following conditions are equivalent:
  \begin{enumerate}[(a)]
  \item\label{item:3} We have a factorization
    \begin{equation*}
      \widetilde{\Phi}_{n,d}(x) = \pm P(x)P(-x)
    \end{equation*}
    for some irreducible $P\in \bZ[x]$. 
  \item\label{item:4} The polynomial $\widetilde{\Phi}_{n,d}$ is reducible in $\bZ[x]$.
  \item\label{item:5} Either
    \begin{itemize}
    \item $d=2\text{ or }3~\mathrm{mod}~4$, $4d|n$ and $\frac{n}{4d}$ is odd;      
    \item or $d=1~\mathrm{mod}~4$ divides $n$ and $n\ne 0~\mathrm{mod}~4$. 
    \end{itemize}
  \end{enumerate}
\end{theorem}
\begin{proof}
  There are two qualitatively different cases: $n$ is divisible by 4, or not. Either way though, the reducibility of $\widetilde{\Phi}_{n,d}$ in $\bZ[x]$ entails that of $\Phi_n$ in $\bQ(\sqrt d)[x]$.
  
  To see this, note first that the case $4|n$ is clear: $\widetilde{\Phi}_{n,d}$ is then obtained from $\Phi_n$ by a simple linear substitution and rescaling. In the other case the factors of $\widetilde{\Phi}_{n,d}$ in $\bZ[x]$ cannot be the transforms of the original (irreducible) factors $\phi_{n_o}$ and $\phi_{2n_o}$, because these contain monomials with odd exponents ($n_o$ being odd) so have some non-zero coefficients in $\sqrt d~ \bZ$. It follows, then, that a factor of $\widetilde{\Phi}_{n,d}$ will have a proper, non-empty subset $\{\lambda_i\}$ of the roots of $\Phi_{n_o}$ as its own roots, and similarly for $\Phi_{2n_o}$ (for a set $\{\lambda'_j\}$); the polynomials
  \begin{equation*}
    \prod(x-\lambda_i)
    \quad\text{and}\quad
    \prod(x-\lambda'_j)
  \end{equation*}
  are proper factors of $\Phi_{n_o}$ and $\Phi_{2n_o}$ respectively over $\bQ(\sqrt 5)$.

  In turn, the factorizability of $\phi_n$ over $\bQ(\sqrt d)$ is equivalent to the inclusion
  \begin{equation}\label{eq:qdinqn}
    \bQ(\sqrt d)\subseteq \text{the cyclotomic field }\bQ_n:=\bQ\left(e^{2\pi i/n}\right).
  \end{equation}
  This would make $n$ divisible by the {\it conductor} of $d$:
  \begin{equation}\label{eq:isdiv}
    d=2,3~\mathrm{mod}~4\Rightarrow 4d|n
    \quad\text{and}\quad
    d=1~\mathrm{mod}~4\Rightarrow d|n
  \end{equation}
  (\cite[\S VI.1, Theorem 1.2, Corollary 1.3 and surrounding discussion]{jan-anf}). These conditions are also weaker than those listed under part \Cref{item:5} of the statement, so we can assume, throughout the proof (regardless of which implication is being addressed), that \Cref{eq:isdiv} holds (along with \Cref{eq:qdinqn}).

  Next, we argue that unless $\widetilde{\Phi}_{n,d}$ is irreducible, its root set splits as a disjoint union of exactly two Galois orbits, each minus the other:
  \begin{equation}\label{eq:a-a}
    \text{roots of }\widetilde{\Phi}_{n,d} = A\sqcup -A.
  \end{equation}

  When $4\nmid n$, so that per $d=2,3~\mathrm{mod}~4$ per \Cref{eq:isdiv}, said partition simply holds (i.e. is not contingent):
  \begin{itemize}
  \item On the one hand, for a primitive $n_o^{th}$ root of unity $\zeta$ the root $-\sqrt d~ \zeta$ is not in the Galois orbit of $\sqrt d~ \zeta$ (because $\zeta$ and $-\zeta$ are not conjugates and an automorphism that fixes $\zeta$ also fixes $\sqrt d\in \bQ(\zeta)$).
  \item So that for every other primitive $n_o^{th}$ root $\zeta'$ exactly one of $\pm\sqrt{d}~\zeta'$ will be a conjugate of $\sqrt{d}~\zeta$. 
  \end{itemize}
  For $4|n$, there are two possibilities:
  \begin{itemize}
  \item $\pm\sqrt d~\zeta$ are conjugates, in which case all roots of $\widetilde{\Phi}_{n,d}$ are;
  \item or not, whence \Cref{eq:a-a} because in that case for each $\zeta'$, $\ord(\zeta')=n$ exactly one of $\pm\sqrt{d}~\zeta'$ is a conjugate of $\sqrt{d}~\zeta$. 
  \end{itemize}
  This much already proves \Cref{item:3} and \Cref{item:4} equivalent to each other and also, when $4\nmid n$, to \Cref{item:5}. It thus remains to prove
  \begin{equation*}
    \text{\Cref{item:5}}\Leftrightarrow\text{\Cref{item:3} and/or \Cref{item:4}}
  \end{equation*}
  when $4|n$. The selfsame argument just employed, though, shows that when $4|n$ $\widetilde{\Phi}_{n,d}$ is reducible if and only if the unique Galois automorphism sending a fixed primitive $n^{th}$ root of unity $\zeta$ to $-\zeta$ also negates $\sqrt d\in \bQ(\zeta)$. This, we claim, is equivalent to \Cref{item:5}, i.e. the ration between $n$ and the {\it discriminant}
  \begin{equation*}
    D:=
    \begin{cases}
      4d&\text{if }d=2,3~\mathrm{mod}~4\\
      d&\text{if }d=1~\mathrm{mod}~4\\
    \end{cases}
  \end{equation*}
  of $\bQ(\sqrt d)$ (\cite[\S 1.2.8]{mp-nt}) is odd. To see why, recall from \cite[\S 4.4.1]{mp-nt} that
  \begin{equation*}
    \mathrm{Gal}\left(\bQ(\zeta)/\bQ(\sqrt d)\right)\subset \mathrm{Gal}(\bQ(\zeta)/\bQ)
  \end{equation*}
  is the kernel of a {\it primitive order-2 Dirichlet character modulo $|D|$} in the sense of \cite[\S 2.2.2, discussion preceding (2.2.7)]{mp-nt}. Now, for a Galois element $\alpha$,
  \begin{equation*}
    \zeta\xmapsto{\quad\alpha\quad}-\zeta
    \quad\Longrightarrow\quad
    \sqrt d\xmapsto{\quad\alpha\quad}-\sqrt d
  \end{equation*}
  is equivalent to
  \begin{equation*}
    1+\frac n2\in (\bZ/n)^{\times}\cong \mathrm{Gal}(\bQ(\zeta)/\bQ)
  \end{equation*}
  {\it not} belonging to that kernel; the primitivity of the above-mentioned character then makes this equivalent to $\frac n2$ no longer being divisible by $D$, i.e. to $\frac nD$ being odd. 
\end{proof}

\begin{remark}
  Restricted to $d=2$ and slightly paraphrased, \Cref{th:whenmodphired} specializes back to \cite[Proposition 2.8]{cc-rec}. 
\end{remark}

\section{Binomial rules matrices}\label{se:binrules}

We will recall the construction of the recursion matrices (``rules matrices'') for weights of RS functions, focusing on the case of monomial and binomial RS functions. The general construction is described in \cite[\S III]{cus_rec}, with more details (including a helpful Mathematica program) in \cite{cus_rec-xv}.

Recall that a {\it Hadamard matrix} is a square matrix with $\pm 1$ and mutually orthogonal rows/columns; see e.g. \cite{hw_had} for an overview. It will be useful to introduce the following $2\times 2$ Hadamard matrices:
\begin{equation*}
  H_0:=
  \begin{pmatrix}
    1&\phantom{-}1\\
    1&-1
  \end{pmatrix}
  ,\quad
  H_1:=
  \begin{pmatrix}
    \phantom{-}1&1\\
    -1&1
  \end{pmatrix}.
\end{equation*}

For any matrix $M$ and non-negative integer $k$, denote 
\begin{itemize}
\item by $\pi_k(M)$ (or $\pi_k M$) the matrix obtained by appending $k$ all-0 columns after each column of $M$ (`$\pi$' is meant to stand for `pad');
\item and by $\rho_k(M)$ (or $\rho_k M$) the matrix obtained by rotation all rows of $M$ simultaneously rightward $k$ times. 
\end{itemize}

In the last item: it makes sense to allow {\it integers} $k$ (not necessarily non-negative): a $(-k)$-fold rightward rotation is a $k$-fold {\it leftward} rotation.

The matrix $R(i)$ of the MRS function $(0,i)$ can now be described as
\begin{equation}\label{eq:ri}
  R(i)=
  \begin{pmatrix}
    \pi_{2^{i-1}-1} H_0\\
    \rho_1\pi_{2^{i-1}-1} H_0\\
    \vdots\\
    \rho_{2^{i-1}-1}\pi_{2^{i-1}-1} H_0
  \end{pmatrix}.
\end{equation}
A simple count confirms that this is a $2^i\times 2^i$ matrix. It will be convenient, when it is understood that the matrices under consideration are square and their size is known, to simply write `$\pi$' for whatever $\pi_k$ operation we need to apply. \Cref{eq:ri}, for instance, would be the more easily-readable 
\begin{equation}\label{eq:ri2}
  R(i)=
  \begin{pmatrix}
    \pi H_0\\
    \rho_1\pi H_0\\
    \vdots\\
    \rho_{2^{i-1}-1}\pi H_0
  \end{pmatrix}.
\end{equation}

Consider, next, a {\it binomial} RS function, of the form $(0,j)+(0,i)$ with $i>j\ge 1$. Its corresponding matrix $R(i,j)$ will differ from $R(i)$ in that some of the instances of $H_0$ in \Cref{eq:ri2} will be $H_1$s instead (which ones, depends on $j$). Specifically, the rows in \Cref{eq:ri2} alternate in whether or not an $H_1$ is substituted in batches of $2^{j-1}$:
\begin{equation}\label{eq:rij}
  R(i,j)=
  \begin{pmatrix}
    \pi H_0\\
    \rho_1\pi H_0\\
    \vdots\\
    \rho_{2^{j-1}-1}\pi H_0\\
    \rho_{2^{j-1}}\pi H_1\\
    \vdots\\
    \rho_{2^{j}-1}\pi H_1\\
    \rho_{2^{j}}\pi H_0\\
    \vdots
  \end{pmatrix}.
\end{equation}
Again, this is a $2^i\times 2^i$ matrix. 

\begin{conjecture}\label{cj:ord}
  For $i>j\ge 1$ the eigenvalues of the matrix $R(i,j)$ of \Cref{eq:rij} are all $\sqrt{2}$-scaled roots of unity.

  Moreover, the multiplicative order of $\frac{R(i,j)}{\sqrt 2}$ is precisely $2\cdot \mathrm{lcm}(4,\ i-j,\ i+j)$. 
\end{conjecture}

In particular, the eigenvalues of $R(i,j)$ would have to be $\sqrt{2}$-scaled roots of unity. It is not difficult to see that $R(i,j)^i$ is a $2^i\times 2^i$ Hadamard matrix: the computation is no more difficult than that delivering \cite[Lemma 5.8]{cc-rec}. Although
\begin{itemize}
\item the roots of an $n\times n$ Hadamard matrix are necessarily algebraic integers all of whose Galois conjugates have absolute value $\sqrt{n}$;
\item and some of the literature focuses on producing Hadamard matrices whose eigenvalues are all scaled roots of unity \cite[Theorem 1.4]{ecs_had},
\end{itemize}
that phenomenon is by no means universal:

\begin{example}\label{ex:badhad}
  Consider the $8\times 8$ {\it Sylvester matrix} (\cite[\S 2, equation (1)]{ecs_had})
  \begin{equation*}
    \begin{pmatrix}
      1&\phantom{-}1\\
      1&-1
    \end{pmatrix}^{\otimes 3},
  \end{equation*}
  and interchange its two top rows. This will produce the $8\times 8$ Hadamard matrix
  \begin{equation*}
    \begin{pmatrix}
      1&-1&\phantom{-}1&-1&\phantom{-}1&-1&\phantom{-}1&-1\\
      1&\phantom{-}1&\phantom{-}1&\phantom{-}1&\phantom{-}1&\phantom{-}1&\phantom{-}1&\phantom{-}1\\
      1&-1&-1&\phantom{-}1&\phantom{-}1&-1&-1&\phantom{-}1\\
      1&\phantom{-}1&\phantom{-}1&\phantom{-}1&-1&-1&-1&-1\\
      1&-1&\phantom{-}1&-1&-1&\phantom{-}1&-1&\phantom{-}1\\
      1&\phantom{-}1&-1&-1&-1&-1&\phantom{-}1&\phantom{-}1\\
      1&-1&-1&\phantom{-}1&-1&\phantom{-}1&\phantom{-}1&-1\\
      1&\phantom{-}1&-1&-1&\phantom{-}1&\phantom{-}1&-1&-1
    \end{pmatrix}.
  \end{equation*}
  Its characteristic (and also minimal) polynomial factors as
  \begin{equation*}
    (x^2-4x+8)(x^2+4x+8)(x^4-12x^2+64);
  \end{equation*}
  while the roots of the first two factors are $\sqrt{8}$-scaled order-8 roots of unity, the roots of the last factor are not: their squares are
  \begin{equation*}
    \sqrt{8}\left(\frac 34\pm \frac{\sqrt 7}{4}\right) = \sqrt{8}(\cos\theta\pm i\sin\theta),
  \end{equation*}
  and because its cosine $\frac 34$ is rational but not 0, $\pm 1$, or $\pm\frac 12$ the angle $\theta$ cannot be a rational multiple of $\pi$ \cite[Corollary 3.12]{niv_irr}. 
\end{example}

\Addresses


\begin{thebibliography}{10}

\bibitem{amt}
Nurdag\"{u}l Anbar, Wilfried Meidl, and Alev Topuzo\u{g}lu.
\newblock Idempotent and {$p$}-potent quadratic functions: distribution of
  nonlinearity and co-dimension.
\newblock {\em Des. Codes Cryptogr.}, 82(1-2):265--291, 2017.

\bibitem{carl}
C.~Carlet.
\newblock {\em Boolean Functions for Cryptography and Coding Theory}.
\newblock Cambridge University Press, 2020.

\bibitem{cc-rec}
Alexandru Chirvasitu and Thomas Cusick.
\newblock Symbolic dynamics and rotation symmetric {B}oolean functions.
\newblock {\em Cryptogr. Commun.}, 14(5):1091--1115, 2022.

\bibitem{cc-bal}
Alexandru Chirvasitu and Thomas~W. Cusick.
\newblock Affine equivalence for quadratic rotation symmetric {B}oolean
  functions.
\newblock {\em Des. Codes Cryptogr.}, 88(7):1301--1329, 2020.

\bibitem{cox-gal}
David~A. Cox.
\newblock {\em Galois theory}.
\newblock Pure and Applied Mathematics (Hoboken). John Wiley \& Sons, Inc.,
  Hoboken, NJ, second edition, 2012.

\bibitem{cus_aff-cub}
Thomas~W. Cusick.
\newblock Affine equivalence of cubic homogeneous rotation symmetric functions.
\newblock {\em Inform. Sci.}, 181(22):5067--5083, 2011.

\bibitem{cus_rec-xv}
Thomas~W. Cusick.
\newblock Weight recursions for any rotation symmetric boolean functions, 2017.
\newblock http://arxiv.org/abs/1701.06648v1.

\bibitem{cus_rec}
Thomas~W. Cusick.
\newblock Weight recursions for any rotation symmetric {B}oolean functions.
\newblock {\em IEEE Trans. Inform. Theory}, 64(4, part 2):2962--2968, 2018.

\bibitem{cs-bk}
Thomas~W. Cusick and Pantelimon St\u{a}nic\u{a}.
\newblock {\em Cryptographic {B}oolean functions and applications}.
\newblock Elsevier/Academic Press, London, second edition, 2017.

\bibitem{ecs_had}
Ronan Egan, Padraig \'{O}~Cath\'{a}in, and Eric Swartz.
\newblock Spectra of {H}adamard matrices.
\newblock {\em Australas. J. Combin.}, 73:501--512, 2019.

\bibitem{hw_had}
A.~Hedayat and W.~D. Wallis.
\newblock Hadamard matrices and their applications.
\newblock {\em Ann. Statist.}, 6(6):1184--1238, 1978.

\bibitem{ir-nt}
Kenneth Ireland and Michael Rosen.
\newblock {\em A classical introduction to modern number theory}, volume~84 of
  {\em Graduate Texts in Mathematics}.
\newblock Springer-Verlag, New York, second edition, 1990.

\bibitem{jan-anf}
Gerald~J. Janusz.
\newblock {\em Algebraic number fields}, volume~7 of {\em Graduate Studies in
  Mathematics}.
\newblock American Mathematical Society, Providence, RI, second edition, 1996.

\bibitem{kph}
Hyeonjin Kim, Sung-Mo Park, and Sang~Geun Hahn.
\newblock On the weight and nonlinearity of homogeneous rotation symmetric
  {B}oolean functions of degree 2.
\newblock {\em Discrete Appl. Math.}, 157(2):428--432, 2009.

\bibitem{MR1429394}
Rudolf Lidl and Harald Niederreiter.
\newblock {\em Finite fields}, volume~20 of {\em Encyclopedia of Mathematics
  and its Applications}.
\newblock Cambridge University Press, Cambridge, second edition, 1997.
\newblock With a foreword by P. M. Cohn.

\bibitem{mp-nt}
Yuri~Ivanovic Manin and Alexei~A. Panchishkin.
\newblock {\em Introduction to modern number theory}, volume~49 of {\em
  Encyclopaedia of Mathematical Sciences}.
\newblock Springer-Verlag, Berlin, second edition, 2005.
\newblock Fundamental problems, ideas and theories, Translated from the
  Russian.

\bibitem{niv_irr}
Ivan Niven.
\newblock {\em Irrational numbers}.
\newblock The Carus Mathematical Monographs, No. 11. Mathematical Association
  of America; distributed by John Wiley \& Sons, Inc., New York, N.Y., 1956.

\bibitem{hash}
J.~Pieprzyk and C.~Qu.
\newblock Fast hashing and rotation-symmetric functions.
\newblock {\em Journal of Universal Computer Science}, 5:20--31, 1999.

\end{thebibliography}
\end{document}